\documentclass[11pt,reqno]{amsart}

\usepackage{mathrsfs, amsmath,amssymb,amsfonts, amsthm, dsfont}
\newtheorem{theorem}{Theorem}[section]

\newtheorem{proposition}[theorem]{Proposition}
\newtheorem{corollary}[theorem]{Corollary}
\newtheorem{lemma}[theorem]{Lemma}
\newtheorem{remark}[theorem]{Remark}

\numberwithin{equation}{section}

\newcommand{\Z}{{\mathbb Z}}
\newcommand{\hcal}{{\mathcal H}}

\numberwithin{equation}{section}

\newtheorem{theo}{{\sc Theorem}}

\newtheorem{lem}[theo]{{\sc Lemma}}
\newtheorem{prop}[theo]{{\sc Proposition}}

\newcommand{\R}{{\mathbb R}}

\usepackage{epsfig}
\usepackage{amscd}
\usepackage{amsmath, amssymb, amsthm}
\usepackage{graphics}
\usepackage{color}
\usepackage{dsfont}
\newtheorem*{main-theorem}{Main Theorem}
\newtheorem*{old-thm}{Theorem}
\theoremstyle{definition}
\numberwithin{equation}{section}

\def\11{\mathds{1}}

\def\phi{\varphi}
\def\half{{\frac{1}{2}}}

\def\be{\begin{eqnarray*}}
\def\ee{\end{eqnarray*}}
\def\ben{\begin{eqnarray}}
\def\een{\end{eqnarray}}

\def\lll{\left\langle}
\def\rrr{\right\rangle}

\def\L2R{L_{\text{Rest}}^2}

\newcommand{\dcal}{\mathcal{D}}

\begin{document}
\title[Number of nodal  domains and singular points]{Number of nodal  domains and singular points of eigenfunctions of negatively curved surfaces with an isometric involution}
\author{Junehyuk Jung and Steve Zelditch}

\address{Department of Mathematical Science, KAIST, Daejeon 305-701, South Korea}

\email{junehyuk@math.kaist.ac.kr}

\address{Department of Mathematics, Northwestern  University, Evanston, IL 60208, USA}

\email{zelditch@math.northwestern.edu}

\thanks{Research partially supported by NSF grant DMS-1206527. The first author was supported by the National Research Foundation of Korea(NRF) grant funded by the Korea government(MSIP)(No. 2013042157)}
\begin{abstract}

We prove two types of  nodal results for density one subsequences of an orthonormal basis $\{\phi_j\}$  of eigenfunctions of
the Laplacian on a negatively curved compact surface. The first type of result involves the intersections $Z_{\phi_j} \cap H$ of the nodal
set $Z_{\phi_j}$ of $\phi_j$ with a smooth curve $H$. Using recent results on quantum ergodic restriction theorems and
prior results on periods of eigenfunctions over curves, we prove that the number of intersection points tends to infinity
for a density one subsequence of the $\phi_j$, and furthermore that the number of such points where  $\phi_j |_H$ changes sign
tends to infinity.   We also prove that the number of zeros of the normal derivative $\partial_{\nu} \phi_j$
on $H$ tends to infinity, also with sign changes. From these results we obtain a lower bound on the number of nodal domains
of even and odd eigenfunctions on surfaces with an orientation-reversing  isometric involution with non-empty and separating
fixed point set. Using (and generalizing) a geometric argument
of Ghosh-Reznikov-Sarnak, we show that the number of nodal domains of even or odd eigenfunctions tends to infinity
for a density one subsequence of eigenfunctions.

\end{abstract}
\maketitle
\section{Introduction}

Let $(M, g)$ be a compact two-dimensional $C^{\infty}$ Riemannian surface of genus $g \geq 2$,  let $\phi_{\lambda}$ be an $L^2$-normalized eigenfunction of the Laplacian,
\[
\Delta \phi_{\lambda} = - \lambda \phi_{\lambda},
\]
let
\[
Z_{\phi_{\lambda}} = \{x: \phi_{\lambda}(x) = 0\}
\]
be its nodal line. This note is concerned with lower bounds on the number of intersections of $Z_{\phi_{\lambda}}$ with a closed curve $\gamma \subset M$ in the case of negatively curved surfaces. More precisely, we show that for closed curves satisfying a generic asymmetry assumption, the number of intersections tends to infinity for a density one subsequence of the eigenfunctions. We also prove the same result for even eigenfunctions when $(M, g)$ admits an
orientation-reversing  isometric involution $\sigma$ whose  fixed point set $\mbox{Fix}(\sigma)$ is separating,
and when  $\gamma$ is a component of $\mbox{Fix}(\sigma)$.  When combined with some geometric arguments adapted from \cite{gzs} the result implies that the number of nodal domains of even (resp. odd) eigenfunctions tends to infinity for a density one subsequence of the eigenfunctions. At the same time, we show that odd eigenfunctions in the same setting have a growing number of singular points\footnote{Singular points are points $x$ where $\phi(x) = d\phi(x) = 0$}. Aside from the  arithemetic case in \cite{gzs} or some explicitly solvable models such as surfaces of revolution, where one can separate variables to find nodal and singular points, these results appear to be the first to give a class of surfaces where the number of nodal domains and critical points are known to tend to infinity for any infinite sequence of eigenfunctions.

We denote the intersections of the nodal set of $\phi_j$ with a closed  curve $H$ by $Z_{\phi_j} \cap H$. We do
not need to assume that $H$ is connected, but it is a finite union of components. We would like to count the number of intersection points. This presumes that the number is finite, but since our purpose is to obtain lower bounds on numbers of intersection points it represents no loss of generality. We define the number to be infinite if the number of intersection points fails to be finite, e.g. if the curve is an arc of the nodal set.

Our first theorem requires the assumption that the closed curve is asymmetric with respect to the geodesic flow.
The precise definition is that $H$ has zero measure of microlocal reflection symmetry in the sense of Definition
1 of \cite{tz1}.  Essentially this means that the two geodesics with mirror image initial velocities emanating from a point of $H$ almost never return to $H$ at the same time to the same place. For more details we refer to \S \ref{QERasym}.

\begin{theorem}\label{theoN}
Let $(M, g) $ be a   $C^{\infty} $ compact negatively curved surface, and let $H$ be a closed curve which is asymmetric with respect to the geodesic flow. Then for any orthonormal eigenbasis $\{\phi_j\}$ of $\Delta$-eigenfunctions of $(M, g)$, there exists  a density $1$ subset $A$ of $\mathbb{N}$ such that
\[
\left\{\begin{array}{l}
\lim_{\substack{j \to \infty \\ j \in A}} \# \; Z_{\phi_j} \cap H = \infty
\\ \\
\lim_{\substack{j \to \infty \\ j \in A}} \# \;  \{x \in H: \partial_{\nu} \phi_j(x) = 0\} = \infty.
\end{array} \right.
\]
Furthermore, there are an infinite number of zeros where  $\phi_j |_H$ (resp. $\partial_{\nu} \phi_j |_H$) changes sign.
\end{theorem}
In fact, we prove that the number of zeros tends to infinity by proving that the number of sign changes tends to infinity.

Although we state the results for negatively curved surfaces, it is sufficient that $(M, g)$ be of non-positive curvature and have ergodic geodesic flow. Non-positivity of the curvature is used to ensure that $(M, g)$ has no conjugate points and that the estimates on sup-norms of eigenfunctions in \cite{Be} apply. Ergodicity is assumed so that the Quantum Ergodic Restriction (QER) results of \cite{ctz} apply. In fact, this theorem generalizes to all dimensions and all hypersurfaces but since our main results pertain to surfaces we only state the results in this case.

We recall that in  \cite{br}, J. Br\"uning (and Yau, unpublished)  showed that $\hcal^1(Z_{\phi_{\lambda}}) \geq C_g \sqrt{\lambda}$, i.e. the length is bounded below by $C_g \sqrt{\lambda}$ for some constant $C_g > 0$. Our methods do not seem to give quantitative lower bounds on the number of nodal intersections. It is known that the number of nodal intersections in the real analytic case is bounded above by $\sqrt{\lambda}$. Some sharp results on flat tori are given by
Bourgain-Rudnick in \cite{BR}.

In contrast, the singular set is a finite set of points, and in \cite{d}, R. T. Dong gave an upper bound for $\# \Sigma_{\phi_{\lambda}}$. No lower bound is possible because $\Sigma_{\phi_{\lambda}} = \emptyset$ for all eigenfunctions of a generic smooth metric \cite{U}. In \cite{Y}, S. T. Yau posed the problem of showing that the number of critical
points of a sequence of eigenfunctions increases with the eigenvalue.  A counter-example was found by Jakobson-Nadirashvili,
who constructed a metric and a sequence of eigenfunctions with a uniformly bounded number of critical points.
One may still ask if there is {\it some} sequence of eigenfunctions for which the number of critical points tends
to infinity.

\subsection{Nodal intersections and singular points for negatively curved surfaces with an isometric involution}

We now assume that $(M, g)$ has an orientation-reversing  isometric involution
\[
\sigma: M \to M, \;\; \sigma^* g= g, \;\; \sigma^2 = Id, \;\;\; \mbox{Fix}(\sigma) \not= \emptyset,
\]
with separating  fixed point set $\mbox{Fix}(\sigma)$.
We refer to \S \ref{sigma} for background on isometric involutions and to  \cite{ss,cp,cp?} for
 more detailed discussions. By a theorem of Harnack, if $\sigma$ is an orientation-reversing involution with non-empty
fixed point set, then $\mbox{Fix}(\sigma)$ is a finite union of simple closed geodesics. We assume that the union is
a separating set (see \S \ref{sigma}).   Although this result
is usually stated for  hyperbolic metrics, it holds for all  negatively curved metrics. We denote by  $L_{even}^2(M)$  the set of $f \in L^2(M)$ such that $\sigma f=f$ and by $L_{odd}^2(Y)$ the $f$ such that $\sigma f = - f$. We denote by $\{\phi_j\}$ an orthonormal eigenbasis of Laplace eigenfunctions of $L_{even}^2(M)$, resp. $\{\psi_j\}$ for  $L^2_{odd}(M)$.

We further denote by
\[
\Sigma_{\phi_{\lambda}} = \{x \in Z_{\phi_{\lambda}}: d \phi_{\lambda}(x) = 0\}
\]
the  singular set of $\phi_{\lambda}.$ These are special critical points $d \phi_j(x) = 0$ which lie on the nodal set $Z_{\phi_j}$. For generic metrics, the singular set is empty \cite{U}. However for negatively curved surfaces with an isometric involution, odd eigenfunctions $\psi$  always have singular points. Indeed, odd eigenfunctions vanish on $\gamma$
and they have singular points  at $x \in \gamma$ where the normal derivative vanishes, $\partial_{\nu} \psi_j = 0$.

%We denote the singular points of $\phi_j$ on $\gamma$  by $\Sigma_j \cap \gamma$.

\begin{theorem}\label{theoS}
Let $(M, g) $ be a compact negatively curved  $C^{\infty} $ surface with an orientation-reversing isometric involution $\sigma : M \to M$ with $\mbox{Fix}(\sigma)$ separating. Let $\gamma \subset \mbox{Fix}(\sigma)$.    Then for any orthonormal eigenbasis $\{\phi_j\}$ of $L_{even}^2(M)$, resp. $\{\psi_j\}$ of $L^2_{odd}(M)$,  one can find a density $1$ subset $A$ of $\mathbb{N}$ such that
\[\left\{ \begin{array}{l}
\lim_{\substack{j \to \infty \\ j \in A}} \# \; Z_{\phi_j} \cap \gamma = \infty\\ \\
\lim_{\substack{j \to \infty \\ j \in A}} \# \; \Sigma_{\psi_j} \cap \gamma = \infty.
\end{array} \right.\]
Furthermore, there are an infinite number of zeros where  $\phi_j |_H$ (resp. $\partial_{\nu} \psi_j |_H$) changes sign.

\end{theorem}

Note that if $Z_{\phi_j} \cap \gamma$ contains a curve, then tangential derivative of $\phi_j$ along the curve vanishes. Hence together with $\partial_{\nu} \phi_j =0$, we have $d \phi_j(x) = 0$, but this is not allowed by \cite{d}. Therefore $Z_{\phi_j} \cap \gamma$ is a finite set of points.

The statement about $  \# \; Z_{\phi_j} \cap \gamma$ follows from the first part of Theorem \ref{theoN},
and the statement about singular points follows from the second part of Theorem \ref{theoN}.  For odd eigenfuntions
under $\sigma$.  points of $\gamma$ with $\partial_{\nu} \psi_j = 0$ are singular. Thus, existence of an
 orientation-reversing isometric involution with seoarating  fixed point sets   is a mechanism which guarantees that a `large' class
of eigenfunctions have a growing number of singular points.
It would be interesting to find a more general mechanism ensuring that the number of critical points of
a sequence of eigenfunctions tends to infinity for a subsequence of eigenfunctions.
 As mentioned above, the  counter-examples of \cite{JN} show that there exist sequences
of eigenfunctions with a uniformly bounded number of critical points. For the sequences we give, the critical
points are singular and therefore are destroyed by  a small perturbation that breaks the symmetry.

\subsection{Counting nodal domains}

The nodal domains of $\phi$ are the connected components of $M\backslash Z_\phi$.
In a recent article \cite{gzs}, Ghosh-Reznikov-Sarnak have proved  a lower bound
on the number of nodal domains of the even Hecke-Maass $L^2$  eigenfunctions of the Laplacian on the finite
area hyperbolic surface $\mathbb{X}=\Gamma\backslash \mathbb{H}$ for $\Gamma  = SL(2, \Z)$. Their lower bound shows that the number
of nodal domains tends to infinity with the eigenvalue at a certain power law rate.  The proof uses  methods
of $L$-functions of arithmetic automorphic forms to get lower bounds on the number
of sign changes of the even eigenfunctions along the geodesic $\gamma$ fixed
by  the isometric involution $(x, y) \to (-x, y)$ of the surface.  It then uses geometric arguments to
relate the number of these sign changes to the number of nodal domains.  We now combine the geometric arguments of \cite{gzs} (compare Lemma \ref{lem1}) with Theorem \ref{theoS} to show that
the number of nodal domains tends to infinity for a density one subsequence of even (resp. odd) eigenfunctions of
any negatively curved surface with an orientation-reversing  isometric involution as above. Before stating the result, let us review the known results on counting numbers of nodal domains.

Let $\{\phi_j\}_{j \geq 0}$ be an orthonormal eigenbasis of $L^2(M)$ with the eigenvalues $0=\lambda_0 \leq \lambda_1 \leq \lambda_2 \leq \cdots$. According to the Weyl law, we have the following asymptotic
\[
j \sim \frac{Vol(M)}{4\pi}\lambda_j.
\]
Therefore by Courant's general nodal domain theorem \cite{ch53}, we obtain an upper bound for $N(\phi_j)$:
\[
N(\phi_j) \leq j = \frac{Vol(M)}{4\pi}\lambda_j(1+o(1)).
\]

When $M$ is the unit sphere $S^2$ and $\phi$ is a random spherical harmonics, then
\[
N(\phi) \sim c\lambda_\phi
\]
holds almost surely for some constant $c>0$ \cite{ns}.
However, for an arbitrary Riemannian surface, it is not even known whether one can always find a sequence of eigenfunctions with growing number of nodal domains. In fact, the number of nodal domains does not have to grow with the eigenvalue, i.e. when $M=S^2$ or $T^2$, there exist eigenfunctions with arbitrarily large eigenvalues with $N(\phi) \leq 3$ (\cite{astern}, \cite{lewy}).
It is conjectured (T. Hoffmann-Ostenhof \cite{H})  that for any Riemannian manifold, there exists a sequence of eigenfunctions
$\phi_{j_k}$ with $N(\phi_{j_k}) \to \infty$. At the present time, this is not even known to hold
for generic metrics.  The results of \cite{gzs} and of the present article are apparently the
first to prove this conjecture for any metrics apart from surfaces of revolution or other
metrics for which separation of variables and exact calculations are possible.

We now recall the result of \cite{gzs}.
Let $\phi$ be an even Maass-Hecke $L^2$ eigenfunction on $\mathbb{X}=SL(2,\mathbb{Z})\backslash \mathbb{H}$. In \cite{gzs}, the number of nodal domains which intersect a compact geodesic segment $\beta \subset \delta=\{iy~|~y>0\}$ (which we denote by $N^\beta(\phi)$) is studied.
\begin{theorem}[\cite{gzs}]
Assume $\beta$ is sufficiently long and assume the Lindelof Hypothesis for the Maass-Hecke $L$-functions. Then
\[
N^\beta(\phi) \gg_\epsilon \lambda_\phi^{\frac{1}{24}-\epsilon}.
\]
\end{theorem}
If one allows possible exceptional set of $\phi$, as an application of Quantitative Quantum Ergodicity and Lindelof Hypothesis on average, one has the following unconditional result.
\begin{theorem}[\cite{jung3}]
Let $\beta \subset \delta$ be any fixed compact geodesic segment. Then within the set of even Maass-Hecke cusp forms in $\{\phi~|~T<\sqrt{\lambda_\phi} <T+1\}$, all but $O(T^{5/6+\epsilon})$ forms satisfy
\[
N^\beta(\phi) > \lambda_\phi^{\frac{\epsilon}{4}}.
\]
\end{theorem}

%The main input to count the nodal domains in \cite{gzs} and \cite{jung3} is the involution on $\mathbb{X}$ induced from $\sigma:x+iy \mapsto -x+iy$ on $\mathbb{H}$. With this symmetry, one can relate the number of nodal domains which intersect $\beta$ with the number of sign changes on $\beta$, because $\beta$ is fixed by $\sigma$ (Theorem 2.2 \cite{gzs} and Lemma \ref{lem1}). Then the number of sign changes on $\beta$ can be studied by estimating various $L^p$-norm for the restriction to $\beta$.

We generalize these results to negatively curved surface with an orienting-reversing isometric  involution $\sigma$ wtih
$\mbox{Fix}(\sigma)$ a separating set (possibly with zero density set of exceptional eigenfunctions.)
\begin{theorem}\label{theo1}
Let $(M,g)$ be a compact negatively curved $C^\infty$ surface with an orientation-reversing  isometric involution $\sigma : M \to M$ with
$\mbox{Fix}(\sigma)$ separating. Assume that $M$ has ergodic geodesic flow.  Then for any orthonormal eigenbasis $\{\phi_j\}$ of $L_{even}^2(Y)$, resp. $\{\psi_j\}$ of $L_{odd}^2(M)$, one can find a density $1$ subset $A$ of $\mathbb{N}$ such that
\[
\lim_{\substack{j \to \infty \\ j \in A}}N(\phi_j) = \infty,
\]
resp.
\[
\lim_{\substack{j \to \infty \\ j \in A}}N(\psi_j) = \infty,
\]

\end{theorem}
\begin{remark}
For odd eigenfunctions, the same conclusion holds with the assumption $\mbox{Fix}(\sigma)$ separating replaced by $Fix(\sigma)\neq \emptyset$.
\end{remark}

Finally, we thank H. Parlier and M. Grohe for helpful comments and references. 

\section{Kuznecov sum formula on surfaces}

We need a prior result \cite{z} on the asymptotics of the `periods' $\int_{\gamma} f \phi_j ds$ of eigenfunctions
over closed geodesics when $f$ is a smooth function.

\begin{theorem}\label{K}  \cite{z} (Corollary 3.3)  Let $f \in C^{\infty}(\gamma)$. Then
there exists a constant $c>0$ such that,
\[
\sum_{\lambda_j < \lambda}\left|\int_{\gamma} f \phi_j ds\right|^2 = c\left|\int_{\gamma} f ds\right|^2 \sqrt{\lambda} + O_f(1).
\]
\end{theorem}
We only use the principal term and not the remainder estimate here.

A small modification of the proof of Theorem \ref{K} is the following: Let $\partial_{\nu} $ denote the normal
derivative along $\gamma$.

\begin{theorem}\label{KN}  Let $f \in C^{\infty}(\gamma)$. Then
there exists a constant $c>0$ such that,
\[
\sum_{\lambda_j < \lambda}\left|\lambda_j^{-1/2} \int_{\gamma} f \partial_{\nu} \phi_j ds\right|^2 = c\left|\int_{\gamma} f ds\right|^2 \sqrt{\lambda} + O_f(1).
\]
\end{theorem}
The proof is essentially the same as for Theorem \ref{K} except that one takes the normal derivative of the wave
kernel in each variable before integrating over $\gamma \times \gamma$. The normalization makes
$\lambda_j^{-1/2} \partial_{\nu}$ a zeroth order pseudo-differential operator, so that the order of the singularity
asymptotics in (2.9) of \cite{z} are the same. The only change is that the principal symbol is multiplied by
the (semi-classical) principal symbol of $\lambda_j^{-\half} \partial_{\nu}$. If we use Fermi normal coordinates
$(s, y)$ along $\gamma$ with $s$ arc-length along $\gamma$ then $\partial_{\nu} = \partial_y$ along $\gamma$
and its symbol is the dual variable $\eta_+$, i.e. the positive part of $\eta$. Here we assume that $\gamma$ is oriented
and that $\nu$ is a fixed choice of unit normal along $\gamma$, defining the `positive' side.

\begin{proposition} \label{Chebyshev}  There exists a subsequence of eigenfunctions $\phi_j$ of
natural density one so that, for all  $f \in C^{\infty}(\gamma)$,
\begin{equation} \label{EST} \left\{ \begin{array}{l}
\left|\int_{\gamma} f \phi_j ds\right| \\ \\
\lambda_j^{-\half} \left|\int_{\gamma} f \partial_{\nu} \phi_j ds\right|  \end{array} \right.  =O_f( \lambda_j^{-1/4} (\log \lambda_j )^{1/2})
\end{equation}

\end{proposition}

\begin{proof} Denote by $N(\lambda)$ the number of eigenfunctions in $\{j~|~\lambda<\lambda_j<2\lambda \}$. For each $f$, we have by Theorem \cite{z} and   Chebyshev's inequality,

\[
\frac{1}{N(\lambda)}|\{j~|~\lambda<\lambda_j<2\lambda,~\left|\int_{\gamma_i} f \phi_j ds\right|^2 \geq \lambda_j^{-1/2}\log \lambda_j \}| = O_f(\frac{1}{\log \lambda}).
\]
It follows that the  upper density of exceptions to \eqref{EST}  tends to zero. We then choose a countable
dense set $\{f_n\}$ and apply the diagonalization argument of \cite{z2} (Lemma 3)  or \cite{Zw} Theorem 15.5 step (2)) to conclude that there exists a density
one subsequence for which \eqref{EST} holds for all $f \in C^{\infty}(\gamma)$. The same holds for the normal
derivative.

\end{proof}

\section{\label{QERasym} Quantum ergodic restriction theorem for Dirichlet or Neumann data}

QER (quantum ergodic restriction) theorems for Dirichlet data assert the quantum ergodicity of restrictions
$\phi_j |_H$
of eigenfunctions or their normal derivatives to hypersurfaces $H \subset M$.
In this section we review the QER theorem for hypersurfaces of
\cite{tz1}.  It is used in the proof of Theorem \ref{theoN}. As mentioned above, it does not apply
to the restrictions of even functions or normal derivatives of odd eigenfunctions to the fixed point set
of an isometry, and the relevant QER theorem for Cauchy data is explained in \S \ref{CDsect}.

\subsection{Quantum ergodic restriction theorems for Dirichlet data}

Roughly speaking, the QER theorem for Dirichlet data says that restrictions of eigenfunctions
to hypersurfaces $H \subset M$ for $(M, g)$ with ergodic geodesic flow  are quantum ergodic along $H$ as
long as $H$ is asymmetric for the geodesic flow. By this is meant that a  tangent vector $\xi$ to $H$ of length $\leq 1$
is the projection to $T H$ of two unit tangent vectors $\xi_{\pm}$  to $M$. The $\xi_{\pm}   = \xi + r \nu$ where
$\nu$ is the unit normal to $H$ and $|\xi|^2 + r^2 = 1$. There are two possible signs of $r$ corresponding to
the two choices of ``inward'' resp. ``outward" normal. Asymmetry of $H$ with respect to the geodesic flow
$G^t$ means that the two orbits $G^t(\xi_{\pm})$ almost never return at the same time to the same place on $H$.
A generic hypersurface is asymmetric. The fixed point set of an isometry $\sigma$ of course fails to be asymmetric
and is the model for a ``symmetric" hypersurface. We refer to \cite{tz1} (Definition 1) for the precise definition
of ``positive measure of microlocal reflection symmetry" of $H$. By asymmetry we mean that this measure is zero.

We now state the special cases relevant to Theorem \ref{theoN}.  We also write $h_j = \lambda_j^{-\half}$
and employ the calculus of semi-classical pseudo-differential operators \cite{Zw} where the
pseudo-differential
operators on $H$ are denoted by   $a^w(y, h D_y)$  or  $Op_{h_j}(a)$. The unit co-ball bundle of $H$ is denoted
by $B^* H$.

  \begin{theorem} \label{sctheorem} Let $(M, g)$ be a compact surface  with ergodic geodesic flow, and let  $H \subset
      M$ be a closed curve which is {\it asymmetric} with respect  to the geodesic flow.  Then
 there exists a  density-one subset $S$ of ${\mathbb N}$ such that
  for $a \in S^{0,0}(T^*H \times [0,h_0)),$
$$ \lim_{j \rightarrow \infty; j \in S} \langle Op_{h_j}(a)
 \phi_{h_j}|_{H},\phi_{h_j}|_{H} \rangle_{L^{2}(H)} = \omega(a), $$
 where
 $$  \omega(a) = \frac{4}{ vol(S^*M) } \int_{B^{*}H}  a_0( s, \sigma )  \,  (1 - |\sigma|^2)^{-\half}  \, ds d\sigma.$$
In particular this holds for multiplication operators $f$.
\end{theorem}
%We note that Theorem \ref{maintheorem} is indeed a special case  of Theorem \ref{sctheorem} with  $a_{-j} = 0; j\geq 1$ and $a_0 \in S^{0}_{cl}(T^*H).$

There is a similar result for normalized Neumann data.
  The normalized  Neumann  data of an eigenfunction along $H$ is denoted by
 \begin{equation}
\lambda_j^{-\half} D_{\nu} \phi_j |_{H}.
\end{equation}
Here, $ D_{\nu} = \frac{1}{i} \partial_{\nu}$ is a fixed choice of unit normal derivative.

We   define the microlocal lifts of the
Neumann data as the  linear functionals on semi-classical symbols $a \in S^{0}_{sc}(H)$ given by
$$\mu_h^N(a): = \int_{B^* H} a  \, d\Phi_h^N : = \langle Op_{H}(a) h D_{\nu} \phi_h
|_{H}, h D_{\nu} \phi_h |_{H}\rangle_{L^2(H)}.  $$

 \begin{theorem} \label{ND} \label{sctheoremNeu} Let $(M, g)$ be a compact surface  with ergodic geodesic flow, and let  $H \subset
      M$ be a closed curve which is {\it asymmetric} with respect  to the geodesic flow.  Then
 there exists a  density-one subset $S$ of ${\mathbb N}$ such that
  for $a \in S^{0,0}(T^*H \times [0,h_0)),$
$$ \lim_{h_j \rightarrow 0^+; j \in S} \mu_h^N(a) \to  \omega(a), $$
 where
 $$  \omega(a) = \frac{4}{ vol(S^*M) } \int_{B^{*}H}  a_0( s, \sigma )  \,  (1 - |\sigma|^2)^{\half}  \, ds d\sigma.$$
In particular this holds for multiplication operators $f$.
\end{theorem}

\section{Proof of Theorem \ref{theoN}}

\subsection{A Lemma}

Define the natural density of a set $A \in \mathbb{N}$ by
\[
\lim_{X\to \infty } \frac{1}{X}|\{x\in A~|~ x<X\}|
\]
whenever the limit exists. We say ``almost all" when corresponding set $A \in \mathbb{N}$ has the natural density $1$. Note that intersection of finitely many density $1$ set is a density $1$ set.
When the limit does not exist we refer to the $\limsup$ as the upper density and the
$\liminf$ as the lower density.

\begin{lemma}\label{lem2}
Let $a_n$ be a sequence of real numbers such that for any fixed $R>0$, $a_n>R$ is satisfied for almost all $n$. Then there exists a density $1$ subsequence $\{a_n\}_{n\in A}$ such that
\[
\lim_{\substack{n\to \infty \\ n \in A} }a_n = +\infty.
\]
\end{lemma}
\begin{proof}
Let $n_k$ be the least number such that for any $n \geq n_k$,
\[
\frac{1}{n}|\{j \leq n~|~a_j>k \}| > 1- \frac{1}{2^k}.
\]
Note that $n_k$ is nondecreasing, and $\lim_{k\to \infty}n_k = +\infty$.

Define $A_k \subset \mathbb{N}$ by
\[
A_k = \{n_k \leq j < n_{k+1}~|~ a_j>k\}.
\]
Then for any $n_k\leq m <n_{k+1}$,
\[
\{j\leq m~|~a_j>k\} \subset \bigcup_{l=1}^k A_l \cap [1,m],
\]
which implies by the choice of $n_k$ that
\[
\frac{1}{m}|\bigcup_{l=1}^k A_l \cap [1,m]| >1- \frac{1}{2^k}.
\]
This proves
\[
A=\bigcup_{k=1}^\infty A_k
\]
is a density $1$ subset of $\mathbb{N}$, and by the construction we have
\[
\lim_{\substack{n\to \infty \\ n \in A} }a_n = +\infty.
\]
\end{proof}

\subsection{Completion of the proof of Theorem \ref{theoN}}

\begin{proof}
Fix $R \in \mathbb{N}$. Let $\gamma_1, \cdots, \gamma_R$ be a partition of the closed curve $H$ and let $\beta_i \subset \gamma_i$ be proper subsegments. Let $f_1, \cdots, f_R \in C_0^\infty (H)$ be given such that
\begin{align*}
supp\{f_i\} = \gamma_i\\
f_i \geq 0 \text{ on } H\\
f_i=1 \text{ on } \beta_i.
\end{align*}
 We may assume that the sequence $\{\phi_j\}$ has the quantum restriction property of
Theorem \ref{sctheorem}, which implies that
\[
\lim_{j \to \infty} ||\phi_j||_{L^2(\beta_i)} = B \cdot \mathrm{length}(\beta_j)
\]
for all $j=1,\cdots, R$ for some constant $B>0$. Namely, $B = \int_{-1}^1 (1 - \sigma^2)^{\half} d\sigma.$ Then
\begin{align*}
\int_{\beta_i} |\phi_j| ds &\geq ||\phi_j||_{L^2(\beta_i)}^2 ||\phi_j||^{-1}_{L^\infty (M)}\\
&\gg \lambda_j^{-1/4} \log \lambda_j.
\end{align*}
Here we use the well-known inequality $||\phi_j||_{L^\infty (M)} \ll \lambda_j^{1/4}/ \log \lambda_j$ which follows
from the remainder estimate in the pointwise Weyl law of \cite{Be}.

By Proposition \ref{Chebyshev},
\[
\left|\int_{\gamma_i} f_i \phi_j ds\right| =O_R( \lambda_j^{-1/4} (\log \lambda_j )^{1/2})
\]
is satisfied for any $i=1, \cdots , R$ for almost all $\phi_j$.

Therefore for all sufficiently large $j$, such $\phi_j$ has at least one sign change on each segment $\gamma_i$ proving that $\#Z_{\phi_j} \cap H \geq R$ is satisfied for every $R > 0$ by almost all $\phi_j$. Now we apply Lemma \ref{lem2} with $a_j = \#Z_{\phi_j} \cap H$ to conclude Theorem \ref{theoN}.

The proof for Neumann data is essentially the same, using Theorem \ref{ND} instead of Theorem \ref{sctheorem}.
\end{proof}

\section{ \label{sigma} Surfaces with an orientation-reversing  isometric involution }

We now specialize to  a negatively curved surface of genus $g \geq 2$  with an  orientation-reversing isometric involution with
non-empty fixed point set. To begin with, we recall some of the known results these objects.

Let $\sigma: M \to M$ be an isometric involution. We first distinguish  several  cases. First is the dichotomy  (i) $\sigma $ is orientation
reversing, or  (ii) $\sigma$ is orientation preserving. Our results only pertain to case (i).

In the case of orientation-reversing involutions, Harnack's theorem says that the fixed point set $\mbox{Fix}(\sigma)$
is a disjoint union
\begin{equation} \label{H} H = \gamma_1 \cup \cdots \cup \gamma_k \end{equation} of $0 \leq k \leq g + 1$ simple closed geodesics. We refer to Theorem 1.1 (see also Lemma 3.3)  of \cite{cp?}.
It is possible that $\mbox{Fix}(\sigma) = \emptyset$, i.e. $k = 0$, i.e. there exist orientation-reversing isometric
involutions with empty fixed point sets \cite{P}. We  assume $k \not= 0$.

There is a further dichotomy accordingly as $H$ \eqref{H} is a separating set or not. We assume that it is throughout this
article. Thus $M \backslash H = M_+ \cup M_-$ where $M_+^0 \cap M_-^0 = \emptyset$ (the interiors are disjoint),
where  $\sigma(M_+) = M_-$ and where $\partial M_+ = \partial M_- = H$. Our results at present do not apply to
the non-separating case, although it is possible that one could extend them to many non-separating cases.

In the case $k = 0$,
there does exist a closed geodesic $\gamma$ such that $\sigma(\gamma) = \gamma$.  But as in
Lemma 3.4 of \cite{cp?}, $\sigma$ is the anti-podal map of $\gamma$, i.e. $\sigma$ acts by an angle $\pi$
rotation.

The case of orientation preserving involutions $\sigma \not= id$  is discussed in \cite{ss}. By the Riemann-Hurwitz
relation, $\sigma$ has $k = 2g + 2 - 4j$ different fixed points for some $0 \leq j \leq \half(g + 1)$. When $\sigma$ has fixed points,  it has at least two fixed points.  If $A, B$
are two distinct fixed points and $u$ is a simple geodesic segment from $A$ to $B$ then $u \cup \sigma(u)$ is
a simple closed geodesic of $M$.

\subsection{Eigenfunctions on surfaces with an orientation-reversing  isometric involution}

 We consider singular points of the even, resp. odd,
eigenfunctions of involutions $\sigma$ with $\mbox{Fix}(\sigma) \not= \emptyset. $

We first consider the case of an orientation reversing involution with $\gamma \subset \mbox{Fix}(\sigma)$.
We could take the curve $C$

\begin{lem} \label{sign}  Let $(M, g)$ admit an orientation reversing isometric involution with separating $\mbox{Fix}(\sigma)$
and  $\gamma$  a  geodesic  such that $\gamma \subset \mbox{Fix}(\sigma)$.  Let $\phi_j$ be an even eigenfunction,   and let $x_0 = \gamma(s_0) $ be a zero of $\phi_j |_{\gamma}$. Then
at a regular zero $x_0$,
$\phi_j |_{\gamma}$ changes sign. That is,  if  the even eigenfunction does not change sign at the zero $x_0$ along
$\gamma$, $x_0$ must be a singular point and $Z_{\phi_j}$ locally stays on one side of $\gamma$.
\end{lem}

Indeed, since $\phi$ is even, its normal derivative vanishes everywhere  on $\gamma$. If $\phi$ does not change
sign at $x_0$, then   $\gamma$ is tangent to $Z_{\phi_j}$ at $x_0$,
i.e.   $\frac{d}{ds} \phi_j(\gamma(s)) = 0$,  so that $x_0$ is a singular point.

Next we consider odd eigenfunctions and let  $\psi_j$ be an odd eigenfunction.
The zeros of $ \partial_{\nu} \psi_j$ on $\gamma$ are also singular points of $\psi_j$.

\begin{lem} Let $(M, g)$ admit an orientation reversing isometric involution
and  $\gamma$  a  geodesic  such that $\gamma \subset \mbox{Fix}(\sigma)$.  Let  $\psi_j$ be an odd eigenfunction.

Then the   zeros of  $\partial_{\nu} \psi_j$  on $\gamma$ are intersection points
of the nodal set of $\psi_j$ in $M \backslash \gamma$ with $\gamma$, i.e. point where at least  two
nodal branches cross. \end{lem}

\begin{proof}

If $x_0$ is a singular point, then $\phi_j (x_0) = d \phi_j (x_j) = 0$, so  the zero set of $\phi_{\lambda}$ is similar to that of a spherical
harmonic of degree $k \geq 2$, which consists of $k \geq 2$  arcs meeting at equal angles at $0$.  It follows that at least two
transvese
branches of the nodal set of an odd eigenfunction meet at each singular point on $\gamma$.

\end{proof}

%Next we consider the case of an orientation-reversing involution such that $\sigma(\gamma) = \gamma$ but
%$\gamma \cap \mbox{Fix}(\sigma) = \emptyset. $ As mentioned above, $\sigma$ acts as the anti-podal map
%on $\gamma$.

\section{Proof of Theorem \ref{theoS} }

\subsection{\label{CDsect}  Quantum ergodic restriction theorems for Cauchy data}

Our
application is to the hypersurface $H$  \eqref{H} given by the fixed point set of the isometric
involution $\sigma$.  Such a hypersurface (i.e. curve)  is  precisely the kind ruled out by the
hypotheses of \cite{tz1}. However the quantum ergodic restriction theorem for Cauchy
data in \cite{ctz} does apply and shows that the even eigenfunctions are quantum ergodic
along  $H$, hence along each component $\gamma$.  The statement we use is the following:
\begin{theorem} \label{useful} Assume that $(M, g)$ has an orientation reversing  isometric involution with
separating   fixed point set $H$. Let  $\gamma$ be a component of $H$.  Let
$\phi_{h}$ be the sequence of even ergodic eigenfunctions. Then,

$$\begin{array}{l}
 \lll Op_{\gamma}(a)  \phi_{h} |_{\gamma}, \phi_{h} |_{\gamma}
\rrr_{L^2(\gamma)} \\ \\ \rightarrow_{h \to 0^+} \frac{4}{ 2 \pi \mbox{Area}(M)} \int_{B^*\gamma} a_0(s,\sigma) (1 - | \sigma |^2)^{-1/2} d s d \sigma.
\end{array}$$
In particular, this holds when $Op_{\gamma}(a)$ is multiplication by a smooth function $f$.

\end{theorem}
We follow \cite{ctz} in using the notation $h_j = \lambda_{\phi}^{-\frac{1}{4}}$ and in dropping
the subscript.
 It also follows  that normal derivatives of odd eigenfunctions are quantum
ergodic along $\gamma$, but we do not use this result here.
We refer to \cite{tz1, ctz} for background and undefined notation for pseudo-differential operators.

We briefly review the results of \cite{ctz} in order to explain how Theorem \ref{useful}  follows from
results on Cauchy data.
  The normalized Cauchy data of an eigenfunction along $\gamma$ is denoted by
 \begin{equation} \label{CD} CD(\phi_h)  := \{(\phi_h |_{\gamma}, \;
h D_{\nu} \phi_h |_{\gamma}) \}.
\end{equation}
Here, $ D_{\nu}$ is a fixed choice of unit normal derivative. The first component of the Cauchy
data is called the Dirichlet data and the second is called the Neumann data.

The QER  result pertains to matrix elements of
semi-classical pseudo-differential operators along $\gamma$ with respect to the restricted
eigenfunctions. We only use multiplication operators in this article but state the background
results for all pseudo-differential operators.
We denote operators on $\gamma$ by  $a^w(y, h D_y)$  or  $Op_{\gamma}(a)$.
We   define the microlocal lifts of the
Neumann data as the  linear functionals on semi-classical symbols $a \in S^{0}_{sc}(\gamma)$ given by
$$\mu_h^N(a): = \int_{B^* \gamma} a  \, d\Phi_h^N : = \langle Op_{\gamma}(a) h D_{\nu} \phi_h
|_{\gamma}, h D_{\nu} \phi_h |_{\gamma}\rangle_{L^2(\gamma)}.  $$
We
also  define the {\it renormalized microlocal lifts} of the Dirichlet
data by
$$\mu_h^D(a): = \int_{B^*\gamma } a \, d\Phi_h^{RD} : = \langle Op_{\gamma}(a) (1 +
h^2 \Delta_{\gamma}) \phi_{h} |_{\gamma},  \phi_{h}|_{\gamma} \rangle_{L^2(\gamma)}.
$$
Here, $h^2 \Delta_{\gamma}$ denotes the negative
tangential Laplacian $- h^2 \frac{d^2}{ds^2} $  for the induced metric on $\gamma$, so that the symbol $1 - |\sigma|^2$ of
the operator $(1+h^2 \Delta_{\gamma})$ vanishes on the tangent directions
 $S^*\gamma$ of $\gamma$.
Finally, we  define the microlocal lift $d \Phi_h^{CD}$ of the Cauchy data  to be the sum
\begin{equation} \label{WIGCD} d \Phi_h^{CD} := d \Phi_h^N + d \Phi_h^{RD}. \end{equation}

 The
first result of \cite{ctz} states   that the Cauchy data of a sequence of quantum ergodic
eigenfunctions restricted to $\gamma$ is
 QER for semiclassical pseudodifferential operators with
symbols vanishing on the glancing set $S^*\gamma$, i.e. that
%  there is a density one subsequence of eigenvalues for
% which
$$ d \Phi_{h}^{CD}  \to \omega, $$
where $$\omega(a)  =  \frac{4}{2 \pi \mbox{Area}
 (M)} \int_{B^*\gamma} a_0(s, \sigma) (1 - | \sigma |^2)^{1/2} d s d \sigma.$$
Here, $B^* \gamma$ refers to the unit ``ball-bundle'' of $\gamma$ (which is the interval
$\sigma \in (-1,1)$ at each point $s$), $s$ denotes arc-length along $\gamma$ and $\sigma$ is the dual
symplectic coordinate.

\begin{theorem}
Assume that $\{\phi_h\}$ is a quantum ergodic sequence of eigenfunctions on $M$.  Then the sequence
$\{d \Phi_{h}^{CD} \}$ \eqref{WIGCD} of microlocal lifts of the Cauchy data of $\phi_h$ is quantum ergodic on $\gamma$ in the sense that for any
 $a \in S^0_{sc}(\gamma),$%  there exists a sub-sequence of eigenvalues
% of density one so that as $ {h_j \to 0^+} $,
$$\begin{array}{l}
\lll Op_H(a) h D_\nu \phi_h |_{\gamma} ,  h D_\nu \phi_h |_{\gamma} \rrr_{L^2(\gamma)}  + \lll Op_{\gamma}(a) (1 +
 h^2 \Delta_{\gamma}) \phi_{h} |_{\gamma}, \phi_{h} |_{\gamma}
\rrr_{L^2(\gamma)} \\ \\ \rightarrow_{h \to 0^+} \frac{4}{\mu(S^*
 M)} \int_{B^*\gamma} a_0(s, \sigma) (1 - | \sigma |^2)^{1/2} d s d\sigma
\end{array}$$
where $a_0$ is the principal symbol of $Op_{\gamma}(a)$.
\end{theorem}

When applied to even eigenfunctions under an orientation-reversing isometric involution with separating fixed
point set, the Neumann data drops out and we get
\begin{corollary} \label{COROLLARY} Let $(M,g)$ have an  orientation-reversing isometric involution with separating fixed point set
$H$ and let $\gamma$ be one of its components.
 Then for any  sequence of  even quantum ergodic eigenfunctions of $(M, g)$,

$$\begin{array}{l}
 \lll Op_{\gamma}(a) (1 +
 h^2 \Delta_{\gamma}) \phi_{h} |_{\gamma}, \phi_{h} |_{\gamma}
\rrr_{L^2(\gamma)} \\ \\ \rightarrow_{h \to 0^+} \frac{4}{\mu(S^*
 M)} \int_{B^*\gamma} a_0(s, \sigma) (1 - | \sigma |^2)^{1/2} d s d\sigma
\end{array}$$

\end{corollary}

This is not the result we wish to apply since we would like to have a limit formula for the
integrals $\int_{\gamma} f \phi_h^2 ds$. Thus we wish to consider the
the microlocal lift $d \Phi_h^D \in
\dcal'(B^* \gamma)$ of the Dirichlet data of $\phi_h$,
$$\int_{B^* \gamma} a \, d\Phi^D_h : = \langle Op_{\gamma}(a) \phi_h|_{\gamma}, \phi_h|_{\gamma}
\rangle_{L^2(\gamma)}. $$ In order to obtain a quantum ergodicity result for the Dirichlet data we need
to introduce the renormalized microlocal lift of the Neumann data,
$$\int_{B^* \gamma} a \, d\Phi^{RN}_h : = \langle (1 + h^2\Delta_{\gamma} +
i0)^{-1} Op_{\gamma}(a) h D_{\nu}\phi_h|_{\gamma}, h D_{\nu}\phi_h |_{\gamma} \rangle_{L^2(\gamma)}. $$

% \cs we should add statement of Thm 2 with $(I+h^2 \Delta_H +i 0)^{-1}$ passed through and general symbols here  \cs

\begin{theorem} \label{thm2}

Assume that $\{\phi_h\}$ is a quantum ergodic sequence on $M$.  Then,  there exists a sub-sequence
of density one as $h \to 0^+$ such that for all  $a \in S^{0}_{sc}(\gamma)$,
\begin{align*}
&\left< (1 + h^2 \Delta_{\gamma} + i0)^{-1} Op_{\gamma}(a)  h D_\nu \phi_h |_{H} , h D_\nu
\phi_h |_{\gamma} \right>_{L^2(\gamma)} + \left< Op_{\gamma}(a)  \phi_{h} |_{\gamma}, \phi_{h} |_{\gamma}
\right>_{L^2(\gamma)} \\
&\rightarrow_{h \to 0^+} \frac{4}{ 2 \pi \mbox{Area}(M)} \int_{B^*\gamma} a_0(s,\sigma) (1 - | \sigma |^2)^{-1/2} d s d \sigma.
\end{align*}
 \end{theorem}

Theorem \ref{useful} follows from Theorem \ref{thm2} since the Neumann term drops out (as before) under
the hypothesis of Corollary \ref{COROLLARY}.

\subsection{Proof of Theorem \ref{theoS}}

The proof of Theorem \ref{theoS} is now the same as the proof of Theorem \ref{theoN}, using Theorem \ref{useful}
in place of Theorem \ref{sctheorem}.

\section{ Local structure of nodal sets in dimension two}

As background for the proof of Theorem \ref{theo1},  we review the
 local structure of nodal sets in dimension two.
%The results are based on
%the local rescaling of  Bers \cite{Bers} and Hartmann-Wintner \cite{HW} and have
% been  used by S. Y. Cheng \cite{Ch} to study the local structure of the nodal sets.
%In addition, the nodal set is somewhat like the 1-skeleton of a cell decomposition
%in which the nodal domains are the 2-cells and we review the relevant background from \cite{L}.

\begin{prop} \cite{Bers,HW, Ch} \label{nodal} Assume that $\phi_{\lambda}$ vanishes to order $k$ at
$x_0$. Let $\phi_{\lambda}(x) = \phi_k^{x_0} (x) + \phi^{x_0}_{k +
1} + \cdots$ denote the $C^{\infty}$ Taylor expansion of
$\phi_{\lambda}$ into homogeneous terms in normal coordinates $x$
centered at $x_0$.  Then $\phi_k^{x_0}(x)$ is a Euclidean harmonic
homogeneous polynomial of degree $k$.
\end{prop}

To prove this, one substitutes  the homogeneous expansion into the
equation  $\Delta \phi_{\lambda} = \lambda^2 \phi_{\lambda}$ and
rescales $x \to \lambda x.$
The rescaled eigenfunction is an eigenfunction of the locally
rescaled Laplacian $$ \Delta^{x_0}_{\lambda} : = \lambda^{-2}  D_{\lambda}^{x_0}
\Delta_g (D_{\lambda}^{x_0} )^{-1} = \sum_{j = 1}^n
\frac{\partial^2}{\partial u_j^2} + \cdots $$  in Riemannian
normal coordinates $u$ at $x_0$ but now with eigenvalue $1$.
Since
$\phi(x_0 + \frac{u}{\lambda})$ is, modulo lower order terms, an
eigenfunction of a standard flat Laplacian on $\R^n$, it behaves near a zero as
a sum of homogeneous Euclidean harmonic polynomials.

In dimension 2, a homogeneous harmonic polynomial of degree $N$ is the real or imaginary part of the unique
holomorphic homogeneous polynomial $z^N$ of this degree, i.e. $p_N(r, \theta) = r^N \sin N \theta$. As observed in \cite{Ch},
there exists a $C^1$ local diffeormorphism $\chi$  in a disc around a zero $x_0$ so that $\chi(x_0) = 0$
and so that   $\phi_N^{x_0}  \circ \chi  = p_N.$
It follows that the restriction of $\phi_{\lambda}$ to a curve H is $C^1$ equivalent around a zero to $p_N$ restricted
to $\chi(H)$.
The nodal set of $p_N$ around $0$ consists of $N$ rays, $\{r (\cos \theta, \sin \theta) : r > 0, p_N |_{S^1}(v) = 0\}$.
It follows that the local structure of the nodal set in a small disc around a singular point p  is $C^1$ equivalent to
$N$ equi-angular rays emanating from $p$. We refer to \cite{Ch} for further details.

\subsection{Isometric involutions and inert nodal domains}

We now apply the local results to obtain a lower bound for the number of inert nodal domains in the spirit
of \cite{gzs} Section 2.

Let us briefly summarize the argument in \cite{gzs} for genus zero surfaces. A nodal domain of an even eigenfunction
is  called {\it inert} if  it is  $\sigma$-invariant, in which case it intersects $\gamma$ in a segment. Otherwise it is called
{\it split}. The number of inert nodal domains of $\phi$ is denoted $R_{\phi}$. The number of sign changes of $\phi$
on $\gamma$ is denoted $n_{\phi}$. The main result of section 2 of \cite{gzs} in genus
zero is that $R_{\phi} \geq \half n_{\phi} + 1$.  It is also stated that $R_{\phi} \geq \half n_{\phi} + 1 - g$  in genus $g$
 (Remark 2.2).
The proof starts with the case where the nodal set is regular. In that case, the nodal line emanating from a regular  sign-change
zero on $\gamma$ must intersect $\gamma$ again at another sign-change zero. The nodal lines intersect $\gamma$
orthogonally in the regular case.  Applying $\sigma$ to the curve
produces an inert nodal domain and the inequality follows. The remainder of the proof is to show that when singular
points occur,
$R_{\phi} - \half n_{\phi} + 1$ never increases when arcs between singular points are removed. Hence
$R_{\phi} - \half n_{\phi} + 1$ is $\geq$ the same in the regular case, which is $\geq 0$.
We  note  that the  local characterization of nodal sets rules out the cusped nodal crossing  of Figure 7 of \cite{gzs}  and so we
omit this case from the discussion below.

We now prove the inequality for even (resp. odd) eigenfunctions  in the higher genus case of a Riemann surface
with an orientation-reversing isometric involution with non-empty fixed point set.

\subsection{Graph structure of the nodal set and completion of proof of Theorem \ref{theo1}}
From Proposition \ref{nodal}, we can give a graph structure (i.e. the structure of a one-dimensional CW complex)
 to $Z_{\phi_{\lambda}}$ as follows.
\begin{enumerate}
\item For each embeded circle which does not intersect $\gamma$, we add a vertex.
\item Each singular point is a vertex.
\item If $\gamma \not\subset Z_{\phi_\lambda}$, then each intersection point in $\gamma \cap Z_{\phi_\lambda}$ is a vertex.
\item Edges are the arcs of $Z_{\phi_\lambda}$ ($Z_{\phi_\lambda} \cup \gamma$, when $\phi_\lambda$ is even) which join the vertices listed above.
\end{enumerate}

This way, we obtain a graph  embeded into the surface $M$.  We recall that an embedded graph $G$ in a surface
$M$ is a finite set $V(G)$ of vertices and a finite set $E(G)$ of edges which are simple (non-self-intersecting)
curves in $M$ such that any two distinct edges have at most one endpoint and no interior points in common.
The {\it faces} $f$ of $G$ are the  connected components of $M \backslash V(G) \cup \bigcup_{e \in E(G)} e$.
The set of faces is denoted $F(G)$. An edge $e \in E(G)$ is {\it incident} to $f$ if the boundary of $f$ contains
an interior point of $e$. Every edge is incident to at least one and to at most two faces; if $e$ is incident
to $f$ then $e \subset \partial f$. The faces are not assumed to be cells and the sets $V(G), E(G), F(G)$ are
not assumed to form a CW complex. Indeed the faces of the nodal graph of odd eigenfunctions are nodal domains, which  do not have to be simply connected.  In the even case,  the faces which do not intersect $\gamma$ are nodal domains and
the ones which do are  inert nodal domains which  are cut in two by $\gamma$.

Now let $v(\phi_\lambda)$ be the number of vertices, $e(\phi_\lambda)$ be the number of edges, $f(\phi_\lambda)$ be the number of faces, and $m(\phi_\lambda)$ be the number of connected components of the graph. Then by Euler's formula (Appendix F, \cite{g}),
\begin{equation}\label{euler}
v(\phi_\lambda)-e(\phi_\lambda)+f(\phi_\lambda)-m(\phi_\lambda) \geq 1- 2 g_M
\end{equation}
where $g_M$ is the genus of the surface.

We use this inequality to give a lower bound for the number of nodal domains for even and odd eigenfunctions.
\begin{lemma}\label{lem1}
For an odd eigenfunction $\psi_j$,
\[
N(\psi_j) \geq \#\left(\Sigma_{\psi_j}\cap \gamma\right) +2 - 2g_M,
\]
and for an even eigenfunction $\phi_j$,
\[
N(\phi_j) \geq \frac{1}{2}\#\left(Z_{\phi_j} \cap \gamma\right)+1-g_M.
\]
\end{lemma}
\begin{proof}
\textbf{Odd case.}
For an odd eigenfunction $\psi_j$, $\gamma \subset Z_{\psi_j}$.  Therefore $f(\psi_j)=N(\psi_j)$.
Let $n(\psi_j)=\#\Sigma_{\psi_j}\cap \gamma$ be the number of singular points on $\gamma$. These points correspond to vertices having degree at least $4$ on the graph, hence
\begin{align*}
0&= \sum_{x:vertices} \mathrm{deg}(x) -2e(\psi_j) \\
&\geq 2\left(v(\psi_j)-n(\psi_j)\right)+4 n(\psi_j)-2e(\psi_j).
\end{align*}
Therefore
\[
e(\psi_j)-v(\psi_j) \geq n(\psi_j),
\]
and plugging into \eqref{euler} with $m(\psi_j)\geq 1$, we obtain
\[
N(\psi_j) \geq n(\psi_j) +2 - 2g_M.
\]

\textbf{Even case.} For an even eigenfunction $\phi_j$, let $N_{in}(\phi_j)$ be the number of nodal domain $U$ which satisfies $\sigma U=U$ (inert nodal domains). Let $N_{sp}(\phi_j)$ be the number of the rest (split nodal domains). From the assumption that $Fix(\sigma)$ is separating, inert nodal domains intersect $\mbox{Fix}(\sigma)$ on simple segments, and $\mbox{Fix}(\sigma)$ divides each nodal domain into two connected components. This implies that, because $\gamma\subset \mbox{Fix}(\sigma)$ is added when giving the graph structure, the inert nodal domain may correspond to two faces on the graph, depending on whether the nodal domain intersects $\gamma$ or not. Therefore $f(\phi_j)\leq 2N_{in}(\phi_j)+N_{sp}(\phi_j)$. 

Observe that each point in $Z_{\phi_j} \cap \gamma$ corresponds to a vertex having degree at least $4$ on the graph. Hence by the same reasoning as the odd case, we have
\[
N(\phi_j) \geq N_{in}+\frac{1}{2}N_{sp}(\phi_j) \geq \frac{f(\phi_j)}{2}\geq \frac{n(\phi_j)}{2} +1 - g_M
\]
where $n(\phi_j)=\#Z_{\phi_j} \cap \gamma$.
\end{proof}

Now Theorem \ref{theo1} follows from Theorem \ref{theoS} and Lemma \ref{lem1}.

\end{document}